\documentclass[10pt]{amsart}
\usepackage{amscd}
\usepackage{amsmath}
\usepackage{amsfonts}
\usepackage{amssymb}
\usepackage{amsthm}
\usepackage{enumerate}
\usepackage[T1]{fontenc}
\usepackage{hyperref}
\usepackage{mathrsfs}
\usepackage{stackrel}
\usepackage[all]{xy}
\usepackage{yhmath}

\newtheorem{theorem}{Theorem}[section]
\newtheorem{lemma}[theorem]{Lemma}
\newtheorem{proposition}[theorem]{Proposition}
\newtheorem{corollary}[theorem]{Corollary}
\theoremstyle{definition}
\newtheorem{definition}[theorem]{Definition}

\newtheorem*{example}{Example}

\DeclareMathOperator{\Sp}{Sp}
\DeclareMathOperator{\im}{Im}

\DeclareMathOperator{\coim}{Coim}

\DeclareMathOperator{\Der}{Der}

\begin{document}
\title[Extending meromorphic connections]{Extending meromorphic connections to coadmissible $\wideparen{\mathcal{D}}$-modules}
\author{Thomas Bitoun}
\address{Department of Mathematics and Statistics, University of Calgary, 2500 University Drive NW, Calgary, AB, Canada, T2N 1N4}
\email{thomas.bitoun@ucalgary.ca}
\author{Andreas Bode}
\address{\'Ecole normale sup\'erieure de Lyon, site Monod, UMPA, 46 all\'ee d'Italie, 69364 Lyon, France}
\email{andreas.bode@ens-lyon.fr}
\maketitle
\begin{abstract}
We investigate when a meromorphic connection on a smooth rigid analytic variety $X$ gives rise to a coadmissible $\wideparen{\mathcal{D}}_X$-module, and show that this is always the case when the roots of the corresponding $b$-functions are all of positive type. \\
We also use this theory to give an example of an integrable connection on the punctured unit disk whose pushforward is not a coadmissible module.
\end{abstract}

\section{Introduction}
Let $K$ be a complete nonarchimedean field with non-trivial valuation, of characteristic zero. The study of $\wideparen{\mathcal{D}}$-modules on rigid analytic $K$-varieties was initiated by Ardakov--Wadsley in \cite{Ardakov1}, \cite{Ardakov2}, see also \cite{Ardakovequiv}, \cite{Bode} for further results. In those papers, the notion of \emph{coadmissibility} is investigated as the natural analogue of coherence in the classical theory. \\
\\
Developing a notion of holonomicity turns out to be more subtle: there is currently no satisfactory theory of characteristic varieties in the rigid analytic setting, and modules satisfying the familiar condition of the vanishing of certain Ext groups (called `weakly holonomic' modules in Ardakov's Oberwolfach Report \cite{Oberwolfach}) still might display some undesirable properties (infinite length, a direct image or inverse image which is not even coadmissible). A study of weakly holonomic $\wideparen{\mathcal{D}}$-modules, their behaviour under operations and some pathologies is presented in \cite{Dcapthree} by Ardakov, Wadsley and the second author.\\
\\
In this paper, we show that the notion of weakly holonomic $\wideparen{\mathcal{D}}$-module cannot be refined to give a more suitable category without losing some integrable connections: we present an example of an integrable connection on the punctured unit disk whose direct image is not coadmissible.\\
We thus answer a question of Ardakov in \cite{Oberwolfach} in the negative. 
\begin{theorem}
\label{introlambda}
Let $X=\Sp K\langle x\rangle$, $j: U\to X$ the embedding of $U=X\setminus\{0\}$, and write $\partial=\frac{\mathrm{d}}{\mathrm{d}x}$. Let $\mathcal{M}_{\lambda}=\wideparen{\mathcal{D}}_U/\wideparen{\mathcal{D}}_U(\lambda-x\partial)=\mathcal{O}_Ux^{\lambda}$ for some $\lambda \in K$ of type zero. \\
Then $j_*\mathcal{M}_{\lambda}$ is not a coadmissible $\wideparen{\mathcal{D}}_X$-module.
\end{theorem}
We will discuss the notion of \emph{type}, taken from \cite[Definition 13.1.1]{Kedlaya}, in section 3 of this paper, and also give an explicit example of a type zero number (for which we thank Konstantin Ardakov and Arthur-C\'esar Le Bras).\\
We note that the literature usually highlights differences between scalars of type 1 and those of type less than 1 ($p$-adic Liouville numbers), whereas in our situation the special role of type zero numbers (which one might think of as `extremely Liouville' numbers) is owed to the specific convergence conditions in $\wideparen{\mathcal{D}}$.\\
\\
Throughout the paper we adopt the following, more general framework. Let $X=\Sp A$ be a smooth affinoid $K$-variety with free tangent sheaf and let 
\begin{equation*}
Z=\{f=0\}\subset X
\end{equation*} 
be a hypersurface. Let $\mathcal{O}_X(*Z)$ denote the sheaf of meromorphic functions with singularities along $Z$, so that e.g. $\mathcal{O}_X(*Z)(X)=A[f^{-1}]$. We consider a meromorphic connection $\mathcal{N}$ on $X$ with singularities along $Z$, by which we mean a $\mathcal{O}_X(*Z)\otimes_{\mathcal{O}_X}\mathcal{D}_X$-module which is coherent over $\mathcal{O}_X(*Z)$. Writing again $j: U\to X$ for the embedding of the complement of $Z$, we can view the restriction $\mathcal{M}=\mathcal{N}|_U$ as an integrable connection on $U$, and ask under which conditions $j_*\mathcal{M}$ is a coadmissible $\wideparen{\mathcal{D}}_X$-module. \\
Considering global sections, we have a $\mathcal{D}(X)[f^{-1}]$-module $N=\mathcal{N}(X)$ which is finitely generated over $A[f^{-1}]$, and we are studying the $\wideparen{\mathcal{D}}(X)$-module $\mathcal{M}(U)=\wideparen{\mathcal{D}}(U)\otimes_{\mathcal{D}(X)[f^{-1}]} N$.\\
The theory of $b$-functions (see \cite{Mebkhout}) ensures that $N$ is finitely presented over $\mathcal{D}(X)$, so that certainly $\wideparen{N}=\wideparen{\mathcal{D}}(X)\otimes_{\mathcal{D}(X)} N$ is always coadmissible. \\
Viewing these two tensor products as suitable completions of $N$, we give several equivalent conditions relating the coadmissibility of $j_*\mathcal{M}$ to properties of the canonical morphism 
\begin{equation*}
\wideparen{N}=\wideparen{\mathcal{D}}(X)\otimes_{\mathcal{D}(X)} N\to \wideparen{\mathcal{D}}(U)\otimes_{\mathcal{D}(X)[f^{-1}]} N=\mathcal{M}(U).
\end{equation*}
\\
One sufficient condition is clearly that the two completions are actually isomorphic, and this turns out to be always the case if all the roots of the corresponding $b$-functions are of positive type (see Theorem \ref{sheafextn}).
\begin{theorem}
\label{intropostype}
Let $m_1, \dots, m_k$ be a finite generating set of $\mathcal{N}(X)$ viewed as an $A[f^{-1}]$-module, and let $b_1, \dots, b_k$ denote the corresponding $b$-functions from \cite[Th\'eor\`eme 3.1.1]{Mebkhout}. If all roots of $b_i$ (in an algebraic closure of $K$) are of positive type for each $i$, then the natural morphism
\begin{equation*}
\wideparen{\mathcal{D}}(X)\otimes_{\mathcal{D}(X)} \mathcal{N}(X)\to \mathcal{M}(U)
\end{equation*}
is an isomorphism, and $j_*\mathcal{M}$ is a coadmissible $\wideparen{\mathcal{D}}_X$-module.
\end{theorem}

We remark that the assumption that $X$ is an affinoid with free tangent sheaf is only used to apply \cite[Th\'eor\`eme 3.1.1]{Mebkhout} directly to $X$. If $X$ is an arbitrary smooth rigid analytic $K$-variety, we can pass to a suitable affinoid covering $(X_i)$ to obtain an analogue of Theorem \ref{intropostype}, with the condition on the roots of $b$-functions imposed for each $X_i$.\\
\\
We conclude by considering the explicit case of $\mathcal{M}_{\lambda}=\mathcal{O}_Ux^{\lambda}$ on the punctured unit disk in section 5. In this case, the natural choice of $b$-function has root $\lambda$ (or an integral shift of it), and the sufficient condition above turns out to be also necessary, as we show in Theorem \ref{posiffcoad}.
\begin{theorem}
\label{intropunct}
Let $j: U\to X$ be the embedding of the punctured unit disk. Then $j_*\mathcal{M}_\lambda$ is a coadmissible $\wideparen{\mathcal{D}}_X$-module if and only if $\lambda$ is of positive type.
\end{theorem}
We thus establish the examples in Theorem \ref{introlambda}.\\
\\
This paper has two appendices: in Appendix A, we discuss elementary properties of completed tensor products for locally convex topological modules. In particular, we show that Ardakov--Wadsley's coadmissible tensor product $\widehat{\otimes}$ agrees with the completed tensor product when coadmissible modules are equipped with their canonical Fr\'echet topology.\\
While most results in this appendix are probably well-known to experts, we could not find a reference in the required level of generality. \\
\\
In Appendix B, we show that our definition of positive type (Definition \ref{defpostype}) is consistent with the theory of type in the literature (e.g. \cite[Definition 13.1.1]{Kedlaya}).\\
\\
We would like to thank Konstantin Ardakov for his suggestions and for his continued interest in this work. We also thank Arthur-C\'esar Le Bras for his example of a type zero number.

\section{General setup}
We briefly introduce our geometric setup and recall some terminology from \cite{Ardakov1}.
\subsection{Spaces and sheaves}
Let $R$ be the valuation ring of $K$ consisting of elements with norm $\leq 1$, and let $\pi\in R\setminus\{0\}$ with $|\pi|<1$.\\
Let $X=\Sp A$ be a smooth affinoid $K$-variety, $f\in A$ non-constant, $Z=\{x\in X: |f(x)|=0\}$, $U=X\setminus Z$. Let $j: U\to X$ denote the open embedding.\\
For simplicity, we will assume that the tangent sheaf of $X$ is free, and we write $L=\mathcal{T}_X(X)=\Der_K(A)$.
\begin{definition}[{see \cite[Definitions 3.1, 6.1]{Ardakov1}}]
An $R$-subalgebra $\mathcal{A}\subset A$ is an \textbf{affine formal model} of $A$ if it is an $R$-algebra of topologically finite type such that $\mathcal{A}\otimes_R K\cong A$. \\
If $\mathcal{A}\subset A$ is an affine formal model, we call an $\mathcal{A}$-submodule $\mathcal{L}\subseteq L$ an $(R, \mathcal{A})$-\textbf{Lie lattice} if the following is satisfied:
\begin{enumerate}[(i)]
\item $\mathcal{L}$ is finitely generated as an $\mathcal{A}$-module, and $\mathcal{L}\otimes_R K\cong L$;
\item $\mathcal{L}$ is closed under the Lie bracket on $L$;
\item $\mathcal{A}$ is stable under the natural action of $\mathcal{L}$ on $A$.
\end{enumerate}
If $Y\cong \Sp B$ is an affinoid subdomain of $X$, we say that an affine formal model $\mathcal{B}\subset B$ is $\mathcal{L}$-\textbf{stable} if it contains the image of $\mathcal{A}$ under the natural restriction morphism $A\to B$ and is preserved under the action of $\mathcal{L}$.
\end{definition}
We now fix an affine formal model $\mathcal{A}\subset A$ and a free $(R, \mathcal{A})$-Lie lattice $\mathcal{L}\subset L$ by choosing a free generating set and rescaling suitably. Without loss of generality, we can assume $f\in \mathcal{A}$.\\
We will consider two different sheaves of differential operators in this paper: the \emph{algebraic} differential operators $\mathcal{D}_X$, and its completion $\wideparen{\mathcal{D}}_X$.\\
Note that for any affinoid subdomain $Y=\Sp B$, the commutator Lie bracket gives $\mathcal{T}_X(Y)$ the structure of a $(K, B)$-Lie-Rinehart algebra in the sense of \cite{Rinehart}, so that we can form the (relative) enveloping algebra $U_B(\mathcal{T}_X(Y))$ entirely analogously to the enveloping algebra of a Lie algebra. We refer to the end of this subsection for an explicit description.
\begin{definition}
The sheaf $\mathcal{D}_X$ is the sheaf of $K$-algebras on $X$ defined by
\begin{equation*}
\mathcal{D}_X(Y)=U_B(\mathcal{T}_X(Y))
\end{equation*}
for any affinoid subdomain $Y\cong \Sp B$.
\end{definition} 

In order to define $\wideparen{\mathcal{D}}$, we recall the auxiliary sheaves $\mathcal{D}_n$ on the site of $\pi^n\mathcal{L}$-accessible subdomains. Recall that an admissible open affinoid subset $Y\subset X$ is called a rational subdomain if it is of the form
\begin{equation*}
Y=X\left(\frac{f_1}{f_0}, \dots, \frac{f_n}{f_0}\right):=\{x\in X: |f_i(x)|\leq |f_0(x)|, \ i=1, \dots n\}
\end{equation*}
for some $f_0, f_1, \dots, f_n\in A$ generating the unit ideal. As usual, we simplify $X(\frac{g}{1})$ to $X(g)$ and $X(\frac{1}{g})$ to $X(g^{-1})$ for any $g\in A$.
\begin{definition}[{\cite[Definitions 4.6, 4.8]{Ardakov1}}]
\label{rationalac}
Let $Y$ be a rational subdomain of $X$. If $Y=X$, we say that it is $\mathcal{L}$-accessible in 0 steps. Inductively, if $n\geq 1$ then we say that it is $\mathcal{L}$-accessible in $n$ steps if there exists a chain $Y\subseteq Z\subseteq X$ such that the following is satisfied:
\begin{enumerate}[(i)]
\item $Z\subseteq X$ is $\mathcal{L}$-accessible in $(n-1)$ steps;
\item $Y=Z(g)$ or $Z(g^{-1})$ for some non-zero $g\in \mathcal{O}(Z)$;
\item there is an $\mathcal{L}$-stable affine formal model $\mathcal{C}\subset \mathcal{O}(Z)$ such that $\mathcal{L}\cdot g\subseteq \mathcal{C}$, where
\begin{equation*}
\mathcal{L}\cdot g=\{\phi(g): \ \phi\in \mathcal{L}\}, 
\end{equation*}
where the action of $\mathcal{L}$ on $\mathcal{O}(Z)$ is induced from the restriction map 
\begin{equation*}
\mathcal{T}_X(X)\to \mathcal{T}_X(Z)=\mathrm{Der}_K(\mathcal{O}(Z)).
\end{equation*}
\end{enumerate}
\end{definition}
Roughly speaking, a rational subdomain $Y$ is $\mathcal{L}$-accessible in one step if $Y=X(g)$ or $Y=X(g^{-1})$ for some $g\in A$ which is compatible with $\mathcal{L}$: if $Y=X(g)$, there exists an $\mathcal{L}$-stable model $\mathcal{C}\subset A$ (which by definition satisfies $\mathcal{A}\subset \mathcal{C}$) such that the image $\mathcal{B}$ of $\mathcal{C}\langle g\rangle$ is an affine formal model of $\mathcal{O}(Y)$ and the image of $\mathcal{B} \otimes_{\mathcal{A}}\mathcal{L}$ is an $(R, \mathcal{B})$-Lie lattice in $\mathcal{T}_X(Y)$. If $Y=X(g^{-1})$, the same description holds mutatis mutandis.\\ 
\\
We will be concerned with subdomains which can be obtained by repeating this process iteratively and glueing:\\ 
A rational subdomain $Y\subseteq X$ is said to be $\mathcal{L}$-accessible if it is $\mathcal{L}$-accessible in $n$ steps for some $n\in \mathbb{N}$.\\
An affinoid subdomain $Y$ of $X$ is said to be $\mathcal{L}$-\textbf{accessible} if it is $\mathcal{L}$-admissible and there exists a finite covering $Y=\cup_{j=1}^r Y_j$ where each $Y_j$ is an $\mathcal{L}$-accessible rational subdomain of $X$.\\
A finite covering $\{Y_j\}$ of $X$ by affinoid subdomains is said to be $\mathcal{L}$-accessible if each $Y_j$ is an $\mathcal{L}$-accessible affinoid subdomain of $X$.
Note that any affinoid subdomain is $\pi^n\mathcal{L}$-accessible for sufficiently large $n$ by \cite[Proposition 7.6]{Ardakov1}. For any $n\geq 0$, consider the sheaf of $K$-algebras $\mathcal{D}_n$ on the site of $\pi^n\mathcal{L}$-accessible subdomains, given by
\begin{equation*}
Y\mapsto \widehat{U_{\mathcal{B}}(\mathcal{B}\otimes_{\mathcal{A}} \pi^n\mathcal{L})}\otimes_R K
\end{equation*}
for $Y\cong \Sp B$ with $\pi^n\mathcal{L}$-stable affine formal model $\mathcal{B}\subset B$. \\
Note that as a $B$-module, $\mathcal{D}_n(Y)$ is naturally isomorphic to $B\widehat{\otimes}_A \mathcal{D}(X)$, where $\mathcal{D}(X)$ is equipped with the seminorm whose unit ball is generated by $\mathcal{A}$ and $\pi^n\mathcal{L}$.
\begin{definition}[{see \cite[Definition 9.3]{Ardakov1}}]
The sheaf $\wideparen{\mathcal{D}}_X$ is the sheaf of $K$-algebras on $X$ defined by
\begin{equation*}
\wideparen{\mathcal{D}}_X(Y)= \wideparen{U_B(\mathcal{T}_X(Y))}=\varprojlim \mathcal{D}_n(Y)
\end{equation*}
for any affinoid subdomain $Y\cong \Sp B$.
\end{definition}
We can view $\wideparen{\mathcal{D}}_X(Y)$ as the completion of $\mathcal{D}_X(Y)$ with respect to every submultiplicative seminorm extending the supremum norm on $B$.\\
\\
Moreover, we have the sheaf $\mathcal{O}_X(*Z)$ of meromorphic functions with poles in $Z$, i.e.
\begin{equation*}
\mathcal{O}_X(*Z)(Y)=B[f^{-1}]
\end{equation*}
for any affinoid subdomain $Y\cong \Sp B$.\\
We set $\mathcal{D}_X(*Z)=\mathcal{O}_X(*Z)\otimes_{\mathcal{O}_X} \mathcal{D}_X$, a sheaf of $K$-algebras with the obvious multiplication.\\
\\
For the convenience of the reader, we describe here explicitly some sections of the sheaves we have introduced so far. Let $\partial_1, \dots, \partial_a$ be a free generating set of the Lie lattice $\mathcal{L}\subset L$ as an $\mathcal{A}$-module, and abbreviate
\begin{equation*}
\partial^{\underline{i}}=\partial_1^{i_1}\dots \partial_a^{i_a}
\end{equation*}
for any $\underline{i}=(i_1, \dots, i_a)\in \mathbb{N}^a$, and $|\underline{i}|=i_1+\dots +i_a$.
\begin{enumerate}[(i)]
\item $\mathcal{O}_X(X)=A$.
\item $\mathcal{O}_X(*Z)(X)=A[f^{-1}]$.
\item $\mathcal{O}_X(U)=\varprojlim A\langle \pi^nf^{-1}\rangle$.
\item $\mathcal{D}_X(X)=\{ \text{finite sums}\ \sum_{\underline{i}\in \mathbb{N}^a} g_{\underline{i}} \partial^{\underline{i}}: \ g_{\underline{i}}\in A\}$.
\item $\mathcal{D}_X(*Z)(X)=\{ \text{finite sums} \ \sum_{\underline{i}\in \mathbb{N}^a} g_{\underline{i}} \partial^{\underline{i}}: g_{\underline{i}}\in A[f^{-1}]\}$.
\item $\mathcal{D}_n(X)=\{\sum_{\underline{i}\in \mathbb{N}^a} g_{\underline{i}}\partial^{\underline{i}}: \ g_{\underline{i}}\in A, \ |\pi|^{-n|\underline{i}|}|g_{\underline{i}}|\to 0 \ \text{as} \ |\underline{i}|\to \infty\}$.
\item $\wideparen{\mathcal{D}}_X(X)=\wideparen{U_A(L)}=\{ \sum_{\underline{i}\in \mathbb{N}^a} g_{\underline{i}} \partial^{\underline{i}}: \ g_{\underline{i}}\in A, |\pi|^{-n|\underline{i}|}|g_{\underline{i}}|\to 0 \ \text{as} \ |\underline{i}|\to \infty \ \forall n\geq 0\}$.
\end{enumerate}

\subsection{Fr\'echet--Stein algebras and coadmissibility}
Let $U_n=X(\pi^nf^{-1})$, which is obtained from $X$ by removing a tubular neighbourhood of $Z$. For example, if $X=\Sp K\langle x\rangle$ is the closed unit disk and $f=x$, then $U_n$ is the closed annulus with inner radius $|\pi|^n$. Note that the $U_n$ form an admissible covering of $U$. \\
Since $f\in \mathcal{A}$, we have $\mathcal{L}\cdot f\subseteq \mathcal{A}$ by definition of $(R, \mathcal{A})$-Lie lattice, and thus $\pi^n\mathcal{L}\cdot \pi^{-n}f\subseteq \mathcal{A}$. In particular, $U_n$ is $\pi^m\mathcal{L}$-accessible for any $m\geq n$. We thus obtain $K$-Banach algebras $\mathcal{D}_m(U_n)$ for any $m\geq n$, and $\wideparen{\mathcal{D}}(U)=\varprojlim \mathcal{D}_n(U_n)$.
\begin{definition}[{see \cite[section 3]{Schneider}}]
A topological $K$-algebra $\mathcal{U}$ is a (two-sided) \textbf{Fr\'echet--Stein algebra} if $\mathcal{U}\cong \varprojlim \mathcal{U}_n$, where for each $n$ the following is satisfied:
\begin{enumerate}[(i)]
\item $\mathcal{U}_n$ is a (two-sided) Noetherian Banach $K$-algebra;
\item the morphism $\mathcal{U}_{n+1}\to \mathcal{U}_n$ makes $\mathcal{U}_n$ a flat $\mathcal{U}_{n+1}$-module on both sides and has dense image.
\end{enumerate}
\end{definition}

It was shown in \cite[Theorem 6.4]{Ardakov1} that $\wideparen{\mathcal{D}}(Y)$ is a Fr\'echet--Stein algebra for any affinoid subdomain $Y$. Moreover, $\mathcal{O}_X(U)=\varprojlim \mathcal{O}(U_n)$ and $\wideparen{\mathcal{D}}(U)=\varprojlim \mathcal{D}_n(U_n)$ are both Fr\'echet--Stein algebras.

\begin{definition}[{see \cite[section 3]{Schneider}}]
A left module $M$ over a Fr\'echet--Stein algebra $\mathcal{U}=\varprojlim \mathcal{U}_n$ is called coadmissible if $M\cong \varprojlim M_n$, where for each $n$ the following is satisfied:
\begin{enumerate}[(i)]
\item $M_n$ is a finitely generated $\mathcal{U}_n$-module;
\item the natural morphism $\mathcal{U}_n\otimes_{\mathcal{U}_{n+1}} M_{n+1}\to M_n$ is an isomorphism of $\mathcal{U}_n$-modules.
\end{enumerate}
\end{definition}
Coadmissible modules over a Fr\'echet--Stein algebra form an abelian category, containing all finitely presented modules (see \cite[Corollary 3.5, Corollary 3.4.v)]{Schneider}). Recall from the comment after \cite[Corollary 3.5]{Schneider} that each coadmissible module $M$ over a Fr\'echet--Stein algebra $\mathcal{U}$ is equipped with a canonical Fr\'echet topology. We will abbreviate this by talking about the canonical $\mathcal{U}$-topology of $M$. This naturally includes the case of finitely generated modules over Noetherian Banach $K$-algebras, corresponding to a constant projective system.
\subsection{Localization on $U$}
We need to introduce the notion of completed tensor product. In the case of tensor products over $K$, this is done in \cite{SchneiderNFA}.
\begin{definition}[{see \cite[ section 17.B]{SchneiderNFA}}]
Given two locally convex $K$-vector spaces $V_1$, $V_2$, the \textbf{projective tensor product topology} on $V_1\otimes_K V_2$ is defined by lattices of the form $L_1\otimes_R L_2$, where $L_1$ (resp. $L_2$) runs over all open lattices in $V_1$ (resp. $V_2$). 
\end{definition}
\begin{definition}
A $K$-algebra $\mathcal{U}$ is called a locally convex algebra if it is equipped with a locally convex topology such that the multiplication map $\mathcal{U}\times \mathcal{U}\to \mathcal{U}$ is continuous.\\
A locally convex module over a locally convex algebra $\mathcal{U}$ is a $\mathcal{U}$-module $V$ equipped with a locally convex topology such that the action map $\mathcal{U}\times V\to V$ is continuous.
\end{definition}
\begin{definition}
Let $\mathcal{U}$ be a locally convex $K$-algebra, $V_1$ a locally convex right $\mathcal{U}$-module and $V_2$ a locally convex left $\mathcal{U}$-module. The projective tensor product topology on $V_1\otimes_{\mathcal{U}}V_2$ is induced by the natural surjection $\rho: V_1\otimes_K V_2\to V_1\otimes_{\mathcal{U}} V_2$. The \textbf{completed tensor product} $V_1\widehat{\otimes}_{\mathcal{U}} V_2$ is the Hausdorff completion of $V_1\otimes_{\mathcal{U}} V_2$ with respect to the projective tensor product topology.
\end{definition}
We refer to Appendix A for the usual basic properties of completed tensor products. We in particular verify that the coadmissible tensor product $\wideparen{\otimes}$ from \cite[Lemma 7.3]{Ardakov1} is just a special case of the completed tensor product defined above, so that we can phrase the definition of a coadmissible $\wideparen{\mathcal{D}}$-module (from \cite[Definitions 8.3, 9.4]{Ardakov1}) as follows.
\begin{definition}
Let $Y$ be a smooth rigid analytic $K$-variety. A $\wideparen{\mathcal{D}}_Y$-module $\mathcal{M}$ is \textbf{coadmissible} if there exists an admissible affinoid covering $(Y_i)$ of $Y$ such that for each $i$ the following is satisfied:
\begin{enumerate}[(i)]
\item $\mathcal{M}(Y_i)$ is coadmissible over the Fr\'echet--Stein algebra $\wideparen{\mathcal{D}}_Y(Y_i)$;
\item the natural morphism
\begin{equation*}
\wideparen{\mathcal{D}}_Y(Z)\widehat{\otimes}_{\wideparen{\mathcal{D}}(Y_i)} \mathcal{M}(Y_i)\to \mathcal{M}(Z)
\end{equation*}
is an isomorphism for each affinoid subdomain $Z\subset Y_i$, where $\mathcal{M}(Y_i)$ is equipped with the canonical $\wideparen{\mathcal{D}}(Y_i)$-topology.
\end{enumerate}
\end{definition}
Returning to the set-up of subsection 2.1, we can now use completed tensor products to describe sections of $\wideparen{\mathcal{D}}$ and of coadmissible $\wideparen{\mathcal{D}}$-modules explicitly.
\begin{proposition}
\label{Dcapoverstein}
Let $Y=\Sp B$ be an affinoid subdomain of $X$.
\begin{enumerate}[(i)]
\item There is a natural isomorphism
\begin{equation*}
B\widehat{\otimes}_A \wideparen{\mathcal{D}}_X(X)\to \wideparen{\mathcal{D}}_X(Y)
\end{equation*}
of locally convex $B$-modules.
\item There is a natural isomorphism
\begin{equation*}
B\widehat{\otimes}_A \wideparen{\mathcal{D}}_X(U)\to \wideparen{\mathcal{D}}_X(Y\cap U)
\end{equation*}
of locally convex $B$-modules.
\item Let $\mathcal{M}$ be a coadmissible $\wideparen{\mathcal{D}}_X$-module. There are natural isomorphisms
\begin{equation*}
B\widehat{\otimes}_A \mathcal{M}(X)\cong \wideparen{\mathcal{D}}(Y)\widehat{\otimes}_{\wideparen{\mathcal{D}}(X)}\mathcal{M}(X)\cong  \mathcal{M}(Y)
\end{equation*}
of locally convex $B$-modules.
\item Let $\mathcal{M}$ be a coadmissible $\wideparen{\mathcal{D}}_U$-module. There are natural isomorphisms
\begin{equation*}
B\widehat{\otimes}_A \mathcal{M}(U)\cong \wideparen{\mathcal{D}}(Y\cap U)\widehat{\otimes}_{\wideparen{\mathcal{D}}(U)}\mathcal{M}(U)\cong \mathcal{M}(Y\cap U)
\end{equation*}
of locally convex $B$-modules.
\end{enumerate}
\end{proposition}
\begin{proof}
\begin{enumerate}[(i)]
\item Note that $\wideparen{\mathcal{D}}(Y)\cong \varprojlim \mathcal{D}_n(Y)\cong \varprojlim (B\widehat{\otimes}_A \mathcal{D}_n(X))$, so that the result follows from Lemma \ref{ctpprops}.(iii).
\item Fix a positive integer $k$ such that $Y\cap U_0=V(f^{-1})$ is $\pi^k\mathcal{L}$-accessible. Then $Y\cap U_n$ is $\pi^{n+k}\mathcal{L}$-accessible, and $\wideparen{\mathcal{D}}(Y\cap U)\cong \varprojlim \mathcal{D}_{n+k}(Y\cap U_n)$. As 
\begin{equation*}
\mathcal{D}_{n+k}(Y\cap U_n)\cong \mathcal{O}(Y\cap U_n)\widehat{\otimes}_{\mathcal{O}(U_n)} \mathcal{D}_{n+k}(U_n)
\end{equation*}
and $\mathcal{O}(Y\cap U_n)\cong B\widehat{\otimes}_A \mathcal{O}(U_n)$ by \cite[Proposition 7.1.4/4]{BGR}, it follows from associativity of the completed tensor product (Lemma \ref{assctp}) and Lemma \ref{ctpprops}.(i) that
\begin{equation*}
\mathcal{D}_{n+k}(Y\cap U_n)\cong B\widehat{\otimes}_A \mathcal{D}_{n+k}(U_n),
\end{equation*}
and the result follows from Lemma \ref{ctpprops}.(iii).
\item By definition of coadmissibility, $\mathcal{M}(Y)\cong \wideparen{\mathcal{D}}(Y)\widehat{\otimes}_{\wideparen{\mathcal{D}}(X)} \mathcal{M}(X)$, so that the result follows immediately from (i) by using associativity and Lemma \ref{ctpprops}.(i).
\item Fix $k$ as in (ii). As before, $\wideparen{\mathcal{D}}(U)\cong \varprojlim \mathcal{D}_n(U_n)$ and 
\begin{equation*}
\wideparen{\mathcal{D}}(Y\cap U)\cong \varprojlim \mathcal{D}_{n+k}(Y\cap U_n)
\end{equation*} 
are Fr\'echet--Stein algebras. The isomorphisms 
\begin{equation*}
\mathcal{M}(U)\cong \varprojlim \left(\mathcal{D}_n(U_n)\otimes_{\wideparen{\mathcal{D}}(U_n)}\mathcal{M}(U_n)\right)
\end{equation*} 
and 
\begin{equation*}
\mathcal{M}(Y\cap U)\cong \varprojlim \left(\mathcal{D}_{n+k}(Y\cap U_n)\otimes_{\wideparen{\mathcal{D}}(U_n)}\mathcal{M}(U_n)\right)
\end{equation*}
exhibit these modules as coadmissible over $\wideparen{\mathcal{D}}(U)$ resp. $\wideparen{\mathcal{D}}(Y\cap U)$. Thus
\begin{align*}
\mathcal{M}(Y\cap U)&\cong \varprojlim \left(\mathcal{D}_{n+k}(Y\cap U_n)\otimes_{\wideparen{\mathcal{D}}(U_n)} \mathcal{M}(U_n) \right)\\
& \cong \varprojlim \left(\mathcal{D}_{n+k}(Y\cap U_n) \otimes_{\mathcal{D}_n(U_n)} \left(\mathcal{D}_n(U_n)\otimes_{\wideparen{\mathcal{D}}(U_n)} \mathcal{M}(U_n)\right) \right)\\
& \cong \wideparen{\mathcal{D}}(Y\cap U)\wideparen{\otimes}_{\wideparen{\mathcal{D}}(U)}\mathcal{M}(U).
\end{align*}
Applying Lemma \ref{ctpprops}.(iv) gives $\mathcal{M}(Y\cap U)\cong \wideparen{\mathcal{D}}(Y\cap U)\widehat{\otimes}_{\wideparen{\mathcal{D}}(U)} \mathcal{M}(U)$, and applying (ii) finishes the proof.
\end{enumerate}
\end{proof}
\subsection{A criterion for coadmissibility}
Let $D=\mathcal{D}(X)$. A $D[f^{-1}]$-module $N$ is called a \textbf{meromorphic connection with singularities along $Z$} if it is finitely generated over $A[f^{-1}]$. We will also use the same terminology to refer to the corresponding $\mathcal{D}_X(*Z)$-module.\\
By \cite[Th\'eor\`eme 3.1.1]{Mebkhout}, $N$ is a finitely presented $D$-module, and thus
\begin{equation*}
\wideparen{N}:=\wideparen{\mathcal{D}}(X)\otimes_D N
\end{equation*} 
is a finitely presented, hence coadmissible $\wideparen{\mathcal{D}}(X)$-module. \\
\\
Let $\mathcal{M}$ be the integrable connection on $U$ determined by
\begin{equation*}
U_n\mapsto \mathcal{O}(U_n)\otimes_{A[f^{-1}]} N=\mathcal{D}(U_n)\otimes_{D[f^{-1}]} N.
\end{equation*} 
By \cite[Proposition 6.2]{Ardakov2}, this is a coadmissible $\wideparen{\mathcal{D}}_U$-module, with
\begin{equation*}
\mathcal{D}_m(U_n)\otimes_{\wideparen{\mathcal{D}}(U_n)} \mathcal{M}(U_n)\cong \mathcal{M}(U_n)
\end{equation*}
for any $m\geq n$.\\
Set $M:=\mathcal{M}(U)=\varprojlim \left(\mathcal{D}_n(U_n)\otimes_{D[f^{-1}]} N\right)=\wideparen{\mathcal{D}}(U)\otimes_{D[f^{-1}]} N$. In particular, $M$ is a finitely presented module over the Fr\'echet--Stein algebra $\wideparen{\mathcal{D}}(U)$. \\
\\
The restrictions $\mathcal{D}_n(X)\to \mathcal{D}_n(U_n)$ now induce $\mathcal{D}_n(X)$-module morphisms
\begin{equation*}
\theta_n: \mathcal{D}_n(X)\otimes_D N\to \mathcal{D}_n(U_n)\otimes_{D[f^{-1}]}N=\mathcal{D}_n(U_n)\otimes_{\wideparen{\mathcal{D}}(U)} M,
\end{equation*}
and taking the limit we obtain a morphism
\begin{equation*}
\theta: \wideparen{N}\to M
\end{equation*}
of $\wideparen{\mathcal{D}}(X)$-modules.\\
\\
Equipping $\mathcal{D}_n(X)\otimes_D N$ with the canonical $\mathcal{D}_n(X)$-Banach structure, and $\mathcal{D}_n(U_n)\otimes N$ with the canonical $\mathcal{D}_n(U_n)$-Banach structure, we see that $\theta_n$ is continuous, as any $\mathcal{D}_n(X)$-module morphism whose domain is a finitely generated Banach module is continuous.\\
Thus, equipping $\wideparen{N}$ with its canonical $\wideparen{\mathcal{D}}(X)$-topology and $M$ with its canonical $\wideparen{\mathcal{D}}(U)$-topology, the morphism $\theta$ is continuous.

\begin{lemma}
\label{equivtop}
If $M=\mathcal{M}(U)$ is coadmissible over $\wideparen{\mathcal{D}}(X)$ then its canonical $\wideparen{\mathcal{D}}(X)$-topology is equivalent to its canonical $\wideparen{\mathcal{D}}(U)$-topology.
\end{lemma}
\begin{proof}
Write $T_1$ for the canonical $\wideparen{\mathcal{D}}(X)$-topology on $M$, and $T_2$ for the canonical $\wideparen{\mathcal{D}}(U)$-topology. As the maps 
\begin{equation*}
\mathcal{D}_n(X)\otimes_{\wideparen{\mathcal{D}}(X)} M\to \mathcal{D}_n(U_n)\otimes_{\wideparen{\mathcal{D}}(U)} M
\end{equation*}
are continuous (since the left hand side is finitely generated over $\mathcal{D}_n(X)$ by assumption), passing to the limit shows that the identity map from $(M, T_1)$ to $(M, T_2)$ is a continuous bijection, so by the Open Mapping Theorem for Fr\'echet spaces (see \cite[Corollary 8.7]{SchneiderNFA}), the two topologies are equivalent.
\end{proof}
We recall the following definition.
\begin{definition}
A continuous morphism $\phi: M_1\to M_2$ of locally convex $K$-vector spaces is called \textbf{strict} if the induced morphism $\coim \phi\to \im \phi$ is a homeomorphism.
\end{definition}
If $M_1$ and $M_2$ are Fr\'echet, it follows from the Open Mapping Theorem that $\phi$ is strict if and only if $\im \phi$ is a closed subspace of $M_2$.

\begin{proposition}
\label{surjiffcoad}
The following are equivalent:
\begin{enumerate}[(i)]
\item The map $\theta$ is surjective.
\item The map $\theta$ is strict with respect to the canonical topologies on $\wideparen{N}$ and $M$.
\item $M$ is a coadmissible $\wideparen{\mathcal{D}}(X)$-module.
\item $M$ is a finitely generated $\wideparen{\mathcal{D}}(X)$-module.
\end{enumerate}
\end{proposition}
\begin{proof}
As already mentioned, the continuous maps $\theta_n: \mathcal{D}_n(X)\otimes N\to \mathcal{D}_n(U_n)\otimes N$ ensure that $\theta$ is always continuous with respect to the canonical Fr\'echet topologies.\\
So $\theta$ is strict if and only if its image is closed by the Open Mapping Theorem, but as the image of $N$ is dense in $M$, this happens if and only if $\theta$ is surjective. So (i) is equivalent to (ii).\\
\\
If $\theta$ is a surjection, this realizes $M$ as the quotient of a finitely presented $\wideparen{\mathcal{D}}(X)$-module by a closed submodule, so that $M$ is coadmissible over $\wideparen{\mathcal{D}}(X)$ by \cite[Lemma 3.6]{Schneider}. Thus (i) implies (iii) and (iv).\\
\\
Conversely, if $M$ is coadmissible, the topology on $M$ agrees with its canonical topology as a $\wideparen{\mathcal{D}}(X)$-module by Lemma \ref{equivtop}. It follows from the remark after \cite[Lemma 3.6]{Schneider} that $\theta$ is strict, so (iii) implies (ii).\\
\\
If $M$ is finitely generated over $\wideparen{\mathcal{D}}(X)$, the surjection $\wideparen{\mathcal{D}}(X)^{\oplus r}\to M$ factors through $\wideparen{\mathcal{D}}(U)^{\oplus r}$ and hence is continuous: the restriction $\wideparen{\mathcal{D}}(X)\to \wideparen{\mathcal{D}}(U)$ is naturally continuous, but any map of coadmissible $\wideparen{\mathcal{D}}(U)$-modules is also continuous, again by the remark after \cite[Lemma 3.6]{Schneider}. Thus $M$ is the quotient of a finitely presented $\wideparen{\mathcal{D}}(X)$-module by a closed submodule and hence coadmissible by \cite[Lemma 3.6]{Schneider}. Thus (iv) implies (iii).\\
\\
To summarize, (i) is equivalent to (ii), (i) implies (iii) and (iii) implies (ii), so the first three statements are equivalent. Moreover, (i) implies (iv) and (iv) implies (iii), finishing the proof.
\end{proof}
We note that this argument also implies that the image of $\theta$ is always a coadmissible submodule of $M$ which is dense with respect to the canonical $\wideparen{\mathcal{D}}(U)$-topology.
\subsection{Extension and localization}
Recall that a finitely presented module over a Fr\'echet--Stein algebra is coadmissible. For any smooth rigid analytic $K$-variety $Y$, we can thus define the extension functor
\begin{align*}
E_Y: \{\mathrm{coherent} \ \mathcal{D}_Y\mathrm{-modules}\} &\to \{\mathrm{coadmissible} \ \wideparen{\mathcal{D}}_Y\mathrm{-modules}\}\\
\mathcal{M}& \mapsto \wideparen{\mathcal{D}}_Y\otimes_{\mathcal{D}_Y} \mathcal{M}
\end{align*}
which is exact by \cite[Lemma 4.14]{Bode}.\\
We can view Proposition \ref{surjiffcoad} in terms of this extension functor and the usual restriction and direct image functors: $N$ localizes to a $\mathcal{D}_X(*Z)$-module $\mathcal{N}$ on $X$ which is a coherent $\mathcal{D}_X$-module (see \cite{Mebkhout}). The map $\theta$ is then the morphism $(E_X\mathcal{N})(X)\to (j_*E_Uj^*\mathcal{N})(X)$. We will often be concerned with modules for which this is an isomorphism.

\begin{proposition}
\label{localiso}
Let $N$ be a meromorphic connection with singularities along $Z$, and assume the conditions in Proposition \ref{surjiffcoad} are satisfied. Let $\mathcal{M}=E_Uj^*\mathcal{N}$ be the sheaf given by
\begin{equation*}
Y\mapsto \wideparen{\mathcal{D}}(Y)\otimes_{D[f^{-1}]} N
\end{equation*}
for $Y$ affinoid, as discussed above. Then $j_*\mathcal{M}$ is a coadmissible $\wideparen{\mathcal{D}}_X$-module.
\end{proposition}
\begin{proof}
We have seen in Proposition \ref{surjiffcoad} that $j_*\mathcal{M}(X)=\mathcal{M}(U)$ is a coadmissible $\wideparen{\mathcal{D}}(X)$-module, so it remains to show that the natural morphism
\begin{equation*}
\wideparen{\mathcal{D}}(Y)\wideparen{\otimes}_{\wideparen{\mathcal{D}}(X)} \mathcal{M}(U)\to \mathcal{M}(Y\cap U)
\end{equation*}
is an isomorphism for any affinoid subdomain $Y=\Sp B$ of $X$.\\
By Proposition \ref{Dcapoverstein}, the left hand side is isomorphic to $B\widehat{\otimes}_A \mathcal{M}(U)$, where $\mathcal{M}(U)$ is equipped with the canonical $\wideparen{\mathcal{D}}(X)$-topology, and the right hand side is isomorphic to $B\widehat{\otimes}_A \mathcal{M}(U)$, where $\mathcal{M}(U)$ is equipped with the canonical $\wideparen{\mathcal{D}}(U)$-topology. The desired isomorphism thus follows from Lemma \ref{equivtop}.
\end{proof}

\section{Numbers of positive type}
In many explicit calculations in what follows, it will be crucial to distinguish between scalars which are of positive type and those of type zero.
\begin{definition}
\label{defpostype}
Let $\lambda \in K$. We say $\lambda$ is \textbf{of positive type} if $\lambda \in \mathbb{Z}_{\geq 0}$ or if $\lambda \notin \mathbb{Z}_{\geq 0}$ and there exists some integer $r$ such that
\begin{equation*}
\frac{\pi^{ir}}{\prod_{j=0}^{i-1}(\lambda-j)}\to 0 \ \mathrm{as} \ i\to \infty.
\end{equation*}
If $\lambda$ is not of positive type, we say it is of type zero.
\end{definition}
We show in Appendix B that our notion of positive type is equivalent to the one in \cite[Definition 13.1.1]{Kedlaya}. In particular, this implies the following:
\begin{enumerate}[(i)]
\item If $\lambda \in K$ with $|\lambda|>1$, then $\lambda$ is of positive type, as $|\lambda-j|=|\lambda|$ for each $j$.
\item Any integer is of positive type (see \cite[Proposition 13.1.5]{Kedlaya}).
\item $\lambda$ is of positive type if and only if $\lambda -n$ is for some $n\in \mathbb{Z}$.
\end{enumerate}

\begin{example}
Note that there exist numbers which are not of positive type (this example is due to Le Bras and was communicated to us by Ardakov):\\
For convenience, let $K=\mathbb{Q}_p$. Set $k_1=p$ and define inductively $k_{n+1}=p^{2k_n}$ for $n>1$. Let
\begin{equation*}
\lambda=\sum_{i=1}^{\infty} p^{k_i}
\end{equation*}
and denote by $m_j$ the partial sum $m_j=\sum_{i=1}^j p^{k_i}$. So in particular, 
\begin{equation*}
|\lambda-m_j|=|p|^{k_{j+1}}=|p|^{p^{2k_j}}.
\end{equation*}
But now for any integer $r$,
\begin{equation*}
\left|\frac{p^{rm_j}}{\lambda-m_j}\right|=|p|^{rm_j-k_{j+1}}.
\end{equation*}
But as $rm_j=r\cdot \sum_{i=1}^j p^{k_i}$ and $k_{j+1}=p^{2k_j}$, the absolute values above tend to infinity as $j$ tends to infinity. Thus $\mathrm{type}(\lambda)=0$ (in the sense of Kedlaya, Definition \ref{kedlayadef}), and Lemma \ref{kedlayaseries} implies that $\lambda$ can't be of positive type.
\end{example}

\section{Extensions of meromorphic connections}
We now return to the setup of section 2, so that $X=\Sp A$ is a smooth affinoid with free tangent sheaf $\mathcal{T}_X$, $\mathcal{A}\subset A$ is an affine formal model, and $U\subseteq X$ is the non-vanishing set of some non-constant $f\in \mathcal{A}$. We write $D=\mathcal{D}(X)$, and let $N$ be a $D[f^{-1}]$-module which is finitely generated over $A[f^{-1}]$.\\
We let $\partial_1, \dots, \partial_a$ be a free generating set of the Lie lattice $\mathcal{L}$ inside $L=\mathcal{T}_X(X)$.\\
\\
Let $m_1, \dots, m_k$ be a finite generating set of $N$ as a $A[f^{-1}]$-module. Then by \cite[Th\'eor\`eme 3.1.1]{Mebkhout}, there exists $P_i(s)\in D[s]$ and monic $b_i(s)\in K[s]$ such that
\begin{equation*}
P_i(s) f^{-s}m_i=b_i(s)f^{-s-1} m_i
\end{equation*}  
for each $i=1, \dots, k$. \\
\\
Replacing $m_i$ by $f^{-r}m_i$ for some $r$, we will always assume that $b_i(s)\neq 0$ for any $s\in \mathbb{Z}_{\geq 0}$. In particular,
\begin{equation*}
f^{-j}m_i=\frac{P_i(j-1)\cdot \dots \cdot P_i(0)}{\prod_{s=0}^{j-1} b_i(s)} m_i
\end{equation*}
for any $j\geq 0$, $i=1, \dots, k$, and the $m_i$ form a finite generating set of $N$ as a $D$-module.
\begin{lemma}
\label{boundedprod}
Suppose that all the roots of $b_i$ (in an algebraic closure of $K$) are of positive type. Then for any $n\geq 0$, there exists some positive integer $r$ such that
\begin{equation*}
\pi^{jr} \frac{P_i(j-1)\dots P_i(0)}{\prod_{s=0}^{j-1} b_i(s)}\in U(\pi^n\mathcal{L})
\end{equation*}
for any $j\geq 1$.
\end{lemma}
\begin{proof}
Replacing $\mathcal{L}$ by $\pi^n\mathcal{L}$, it is enough to treat the case $n=0$.\\
Then the Lie lattice $\mathcal{L}$ determines a submultiplicative norm on $D=U_A(L)$, with unit ball $U(\mathcal{L})$, and we are required to show that the given elements have norm less than or equal to $1$.\\
Setting $|s|=1$ for the formal parameter $s$ extends the norm to a norm on $D[s]$ in such a way that for any integer $j\in \mathbb{Z}$, the evaluation map $D[s]\to D$ sending $s$ to $j$ is contracting, i.e. bounded of norm $\leq 1$ (by the triangle inequality, as $|j|\leq 1$). Therefore,
\begin{equation*}
|P_i(j)| \leq |P_i(s)|
\end{equation*}
for any $j\in \mathbb{Z}$.\\
Let $r'\in \mathbb{Z}$ such that $|P_i(s)|\leq |\pi|^{r'}$, then the above shows that
\begin{equation*}
\left|P_i(j-1)\cdot \dots \cdot P_i(0)\right| \leq |\pi|^{jr'}
\end{equation*}
for any $j\geq 1$.\\
\\
Now let $\lambda_1, \dots, \lambda_d$ be the roots of $b_i(s)$, with multiplicity, so that
\begin{equation*}
b_i(s)=(s-\lambda_1) \dots (s-\lambda_d),
\end{equation*}
and hence
\begin{equation*}
\prod_{s=0}^{j-1} b_i(s)= \prod_{s=0}^{j-1} (s-\lambda_1)\dots\prod _{s=0}^{j-1}(s-\lambda_d).
\end{equation*}
For any $t\in \{1, \dots, d\}$, $\lambda_t$ is of positive type by assumption, so there exists some $r''_t\geq 0$ such that
\begin{equation*}
\frac{\pi^{jr''_t}}{\prod_{s=0}^{j-1}(s-\lambda_t)}\to 0
\end{equation*}
as $j\to \infty$. Thus
\begin{equation*}
\frac{\pi^{jr''}}{\prod_{s=0}^{j-1} (s-\lambda_1)\cdot \dots \cdot \prod_{s=0}^{j-1} (s-\lambda_d)}\to 0
\end{equation*} 
as $j\to \infty$ for $r''=r''_1+\dots+r''_d$. In particular the sequence is bounded, and replacing $r''$ by a suitable larger integer, we can assume that these terms have norm less than or equal to $1$ for each $j\geq 1$.
\\
Then any $r\geq r''-r'$ has the desired property.
\end{proof}

\begin{theorem}
\label{generalextn}
Suppose that for each $i$, all the roots of $b_i$ (in an algebraic closure of $K$) are of positive type. Then $\wideparen{\mathcal{D}}(U)\otimes_{D[f^{-1}]} N$ is a coadmissible $\wideparen{\mathcal{D}}(X)$-module, and the morphism $\theta: \wideparen{\mathcal{D}}(X)\otimes_{D} N\to \wideparen{\mathcal{D}}(U)\otimes_{D[f^{-1}]} N$ is an isomorphism.
\end{theorem}
To prove Theorem \ref{generalextn}, we establish the following terminology, similar to Lemma \ref{equivtop}. The module $N$ comes equipped with two different locally convex topologies. Firstly, the surjection $D^k\to N$ obtained from the generating set $\{m_i\}$ above induces the quotient topology, which can be seen as being induced by the semi-norms with unit balls 
\begin{equation*}
\sum_i U(\pi^{n_i}\mathcal{L})m_i
\end{equation*}
for various $n_i\geq 0$. We call this topology $T_1$.\\
Secondly, we can consider in the same way the surjection $D[f^{-1}]^k\to N$, giving a topology $T_2$ induced by semi-norms with unit balls
\begin{equation*}
\sum_i U(\pi^{n_i}\mathcal{L})[\pi^{n_i}f^{-1}]m_i.
\end{equation*}
Note that the completion of $(N, T_1)$ is $\wideparen{N}=\varprojlim (\mathcal{D}_n(X)\otimes_D N)=\wideparen{\mathcal{D}}(X)\otimes_D N$, and the completion of $(N, T_2)$ is $M=\wideparen{\mathcal{D}}(U)\otimes_{D[f^{-1}]}N$.\\
\\
It is now clear from the definition that the identity map $(N, T_1)\to (N, T_2)$ is continuous. The result will follow straightforwardly once we have established strictness.
\begin{lemma}
Suppose that for each $i$, all the roots of $b_i$ are of positive type. Then the morphism
\begin{align*}
\phi: & D^k\to N\\
& (\xi_i)\mapsto \sum_i \xi_im_i
\end{align*}
is a strict surjection when $N$ is equipped with the topology $T_2$.\\
In particular, $T_1$ is equivalent to $T_2$.
\end{lemma} 
\begin{proof}
As we have already determined that the map is a continuous surjection, we want to establish that it is also open.\\
Fix $i$, and let $n\geq 0$. It now suffices to show that $\phi((0, \dots, 0, U(\pi^n\mathcal{L}), 0, \dots, 0))$, which is just $U(\pi^n\mathcal{L})\cdot m_i$, contains a set of the form $U(\pi^m\mathcal{L})[\pi^mf^{-1}]m_i$ for some $m\geq 0$.\\
\\
Now by Lemma \ref{boundedprod}, there exists some positive integer $r$ such that
\begin{equation*}
\pi^{jr} \frac{P_i(j-1) \cdot \dots \cdot P_i(0)}{\prod_{s=0}^{j-1} b_i(s)}\in U(\pi^n\mathcal{L})
\end{equation*}
for any $j\geq 0$.\\
\\
In particular, if $m=n+r$ and $x\in U(\pi^m\mathcal{L})[\pi^mf^{-1}]m_i$, we can write
\begin{equation*}
x=\left(\sum_{\underline{i}\in \mathbb{N}^a, j\geq 0} \pi^{m(|\underline{i}|+j)}b_{\underline{i}, j} \partial^{\underline{i}}  f^{-j}\right) m_i
\end{equation*}
by \cite[Theorem 3.1]{Rinehart}, where $b_{\underline{i}, j}\in \mathcal{A}$ and $\partial^{\underline{i}}=\partial_1^{i_1}\dots \partial_a^{i_a}\in U_{\mathcal{A}}(\mathcal{L})$ by definition. Hence
\begin{equation*}
x= \left(\sum_{\underline{i}, j} b_{\underline{i}, j}\partial^{\underline{i}} \pi^{m(|\underline{i}|+j)} \frac{P_i(j-1)\dots P_i(0)}{\prod_{s=0}^{j-1} b_i(s)} \right) m_i.
\end{equation*}
Now by our choice of $r$ 
\begin{equation*}
\pi^{(n+r)|i|+nj}b_{\underline{i}, j}\partial^{\underline{i}}\cdot \pi^{jr} \frac{P_i(j-1)\dots P_i(0)}{\prod_{s=0}^{j-1} b_i(s)} \in U(\pi^n\mathcal{L}) 
\end{equation*}
for all $\underline{i}\in \mathbb{N}^a$ and all $j\geq 0$, and thus $x\in U(\pi^n\mathcal{L})\cdot m_i$, as required.
\end{proof}

\begin{proof}[Proof of Theorem \ref{generalextn}]
By the above, the identity morphism $(N, T_1)\to (N, T_2)$ is an isomorphism of locally convex vector spaces. Thus their completions are isomorphic, i.e. $\theta$ is an isomorphism and $M$ is a coadmissible $\wideparen{\mathcal{D}}(X)$-module by Proposition \ref{surjiffcoad}.
\end{proof}
We can now prove Theorem \ref{intropostype} from the introduction.
\begin{theorem}
\label{sheafextn}
Let $X=\Sp A$ be a smooth affinoid $K$-variety with free tangent sheaf, $f\in A$ non-constant, $Z=\{f=0\}\subset X$. Write $j: U\to X$ for the embedding of the complement of $Z$. Let $\mathcal{N}$ be a meromorphic connection on $X$ with singularities along $Z$, and let $\mathcal{M}=E_Uj^*\mathcal{N}$ be the corresponding integrable connection on $U$.\\
Let $m_1, \dots, m_k$ be a finite generating set of $\mathcal{N}(X)$ viewed as an $A[f^{-1}]$-module, and let $b_1, \dots, b_k$ denote the corresponding $b$-functions. If all roots of $b_i$ are of positive type for each $i$, then the natural morphism
\begin{equation*}
E_X\mathcal{N}\to j_*\mathcal{M}
\end{equation*}
is an isomorphism, and $j_*\mathcal{M}$ is a coadmissible $\wideparen{\mathcal{D}}_X$-module.
\end{theorem}
\begin{proof}
By Theorem \ref{generalextn}, we know that $E_X\mathcal{N}(X)\to j_*\mathcal{M}(X)$ is an isomorphism of $\wideparen{\mathcal{D}}_X(X)$-modules. In particular, $j_*\mathcal{M}(X)$ is a coadmissible $\wideparen{\mathcal{D}}_X(X)$-module.\\
Thus the conditions in Proposition \ref{surjiffcoad} are satisfied, so that $j_*\mathcal{M}$ is a coadmissible $\wideparen{\mathcal{D}}_X$-module by Proposition \ref{localiso}. Hence \cite[Theorem 8.2]{Ardakov1} implies that the natural morphism
\begin{equation*}
E_X\mathcal{N}\to j_*\mathcal{M}
\end{equation*}
is an isomorphism, as it is an isomorphism on the level of global sections. 
\end{proof}

\section{The modules $\mathcal{M}_{\lambda}$ on the punctured unit disk}
We now discuss a particular family of examples on the punctured unit disk. This will give rise to a collection of modules $N$ for which the conditions in Proposition \ref{surjiffcoad} are not satisfied.\\
\\
Let $X= \Sp A=\Sp K\langle x\rangle$, $U=X\setminus\{0\}$, and $U_n=\Sp K\langle x, \pi^nx^{-1}\rangle$.\\
We write $\partial$ for the free generator $\frac{\mathrm{d}}{\mathrm{d}x}$ of $\mathcal{T}_X(X)$, and let $\mathcal{L}$ be the $R\langle x\rangle$-lattice generated by $\partial$.\\
\\
Fix $\lambda \in K$, and let $N_{\lambda}=A[x^{-1}]\cdot x^{\lambda}$ be equipped with the natural $D[x^{-1}]$-module structure. As before, we obtain a coadmissible $\wideparen{\mathcal{D}}_U$-module $\mathcal{M}_{\lambda}\cong \mathcal{O}_U\cdot x^{\lambda}$, and a morphism of $\wideparen{\mathcal{D}}(X)$-modules $\theta_{\lambda}: \wideparen{N_{\lambda}}\to M_{\lambda}=\mathcal{M}_{\lambda}(U)$.

\begin{proposition}
\label{coadimpliespos}
If $M_{\lambda}$ is a coadmissible $\wideparen{\mathcal{D}}(X)$-module then $\lambda$ is of positive type.
\end{proposition}
\begin{proof}
Suppose $M_{\lambda}$ is a coadmissible $\wideparen{\mathcal{D}}(X)$-module and $\lambda$ is of type zero. If $\lambda$ is an integer, it is of positive type by \cite[Proposition 13.1.5]{Kedlaya}, so we have in particular that $\lambda\notin \mathbb{Z}$.\\
Replacing $\lambda$ by $\lambda-n$ for some integer $n$ does not change the property of being of type zero, and $A[x^{-1}]x^{\lambda}\cong A[x^{-1}]x^{\lambda-n}$ as $D$-modules. In this way, we can assume that $N_{\lambda}$ is generated as a $D$-module by $x^{\lambda}$.\\
Then the $\wideparen{\mathcal{D}}(X)$-submodule of $M_{\lambda}$ generated by $x^{\lambda}$ contains $N_{\lambda}$ and is thus dense with respect to the canonical $\wideparen{\mathcal{D}}(X)$-topology by Lemma \ref{equivtop}. As it is also finitely generated, the same argument as in Proposition \ref{surjiffcoad}, (iv) implies (iii), shows that $\wideparen{\mathcal{D}}(X)\cdot x^{\lambda}$ is coadmissible and hence closed in $M_{\lambda}$ by \cite[Lemma 3.6]{Schneider}. Therefore $M_{\lambda}$ is generated by $x^{\lambda}$ as a $\wideparen{\mathcal{D}}(X)$-module.\\
\\
Let $\epsilon_0=1$, $i_0=j_0=0$. We now pick inductively $\epsilon_r\in \mathbb{R}_{>0}$, $j_r, i_r\in \mathbb{N}$ as follows.\\
As $\lambda$ is not of positive type, there exists some real number $0<\epsilon_r<\epsilon_{r-1}$ such that
\begin{equation*}
\left|\frac{\pi^{2ri}}{\prod_{j=0}^{i-1}(\lambda-j)}\right|>\epsilon_r
\end{equation*}
for infinitely many natural numbers $i$. Without loss of generality, we can take $\epsilon_r$ to be of the form $|\pi|^{rj_r}$ for some natural number $j_r$, and let $i_r>\max\{j_r, i_{r-1}\}$ be a natural number satisfying the inequality above. \\
\\
Now consider the element
\begin{equation*}
m=\sum_{r=1}^{\infty} \pi^{(2i_r-j_r)r}x^{-i_r}\cdot x^{\lambda}\in M_{\lambda}=\mathcal{O}(U)x^{\lambda}.
\end{equation*}
As $|\pi|^{ni_r}|\pi|^{(2i_r-j_r)r}\leq|\pi|^{ni_r+(2i_r-i_r)r}= |\pi|^{(n+r)i_r}$ tends to $0$ for any $n\in \mathbb{Z}$ as $r$ tends to infinity, this is indeed an element of $M_{\lambda}$.\\
\\
As $M_{\lambda}$ is generated by $x^{\lambda}$ as a $\wideparen{\mathcal{D}}(X)$-module, there exist elements $g_j\in K\langle x\rangle$ such that $\sum g_j\partial^j\in \wideparen{\mathcal{D}}(X)$ and
\begin{equation*}
\sum_{j\geq 0} g_j\partial^j\cdot x^{\lambda}=m=\sum_{r\geq 1} \pi^{(2i_r-j_r)r}x^{-i_r}\cdot x^{\lambda}.
\end{equation*}
We now claim that we can assume that the $g_j$ all lie in $K$.\\
\\
Writing $g_j=\sum g_{ij}x^i\in K\langle x\rangle$, we have
\begin{align*}
\sum g_j\partial^j x^{\lambda}&= \sum g_{ij} \lambda \dots (\lambda-j+1)x^{\lambda+i-j}\\
&=\sum_{t\geq 0} \left(g_{0, t}+\sum_{j=t+1}^{\infty} g_{j-t, j} (\lambda-t)\dots(\lambda-j+1)\partial^t\right) \cdot x^{\lambda}
\end{align*}
where we write $t=j-i$ and eliminate all terms with $t<0$ by comparing coefficients with $m$.\\
We will show that $\sum h_t\partial^t\in \wideparen{\mathcal{D}}(X)$, where we have abbreviated
\begin{equation*}
h_t=g_{0, t}+\sum_{j=t+1}^{\infty} g_{j-t, j} (\lambda-t)\dots (\lambda-j+1).
\end{equation*}
First note that $|\lambda|\leq 1$ by assumption, and hence $|\lambda -a|\leq 1$ for any $a\in \mathbb{Z}$.\\
Fix $\epsilon>0$, $n\in \mathbb{N}$. As $\sum g_j\partial^j\in \wideparen{\mathcal{D}}(X)$, we know that there exists some $j'$ such that
\begin{equation*}
|g_{ij}|\cdot |\pi|^{-nj}<\epsilon \ \ \forall i, \ \forall j\geq j'.
\end{equation*}
Thus for any $t\geq 0$, $|g_{j-t, j}|<\epsilon$ for $j\geq j'$, so $h_t$ defines an element in $K$. Moreover, we have
\begin{align*}
|\pi|^{-nt}|h_t|&\leq |\pi|^{-nt}\sup \{|g_{0, t}|, |g_{j-t, j}(\lambda-t)\dots (\lambda-j+1)|: j>t\} \\
& \leq |\pi|^{-nt}\sup \{ |g_{0, t}|, |g_{j-t, j}|: j>t\}.
\end{align*}
Let $t\geq j'$, so that $|\pi|^{-nt} |g_{0, t}|< \epsilon$ and 
\begin{equation*}
|\pi|^{-nt}|g_{j-t, j}|<|\pi|^{-nj} |g_{j-t, j}|<\epsilon
\end{equation*} 
for any $j>t$.\\
Hence $|\pi|^{-nt}|h_t|<\epsilon$ for $t\geq j'$. Thus
\begin{equation*}
|\pi|^{-nt} |h_t|\to 0 \ \text{as} \ t\to \infty
\end{equation*}
for any $n$, and $\sum h_t\partial^t\in \wideparen{\mathcal{D}}(X)$.\\
Thus we can assume that $\sum g_j\partial^j x^{\lambda}=m$, where each $g_j\in K$ and $\sum g_j\partial^j\in \wideparen{\mathcal{D}}(X)$.\\
\\
But now
\begin{equation*}
\sum g_j\partial^j\cdot x^{\lambda}=\sum g_j\lambda \cdot (\lambda-1)\cdot \dots \cdot(\lambda-j+1)x^{\lambda-j},
\end{equation*}
so by comparing coefficients with $m$ we obtain
\begin{equation*}
g_{i_r}=\frac{\pi^{(2i_r-j_r)r}}{\prod_{i=0}^{i_r-1}(\lambda-i)}.
\end{equation*}
Thus
\begin{equation*}
|g_{i_r}|=\left|\frac{\pi^{2ri_r}}{\prod_{i=0}^{i_r-1}(\lambda-i)}\right|\epsilon_r^{-1}>1 
\end{equation*}
for any $r\geq 1$, by construction of the $\epsilon_r=|\pi|^{rj_r}$. But $\sum g_j\partial^j$ was supposed to give an element in $\wideparen{\mathcal{D}}(X)$, which produces the desired contradiction.
\end{proof}
This provides us with the first examples of a module not satisfying the conditions in Proposition \ref{surjiffcoad}, by taking $N=N_{\lambda}$ for $\lambda\in K$ of type zero. In particular, we have established Theorem \ref{introlambda} from the introduction. Combining this with our previous results, we also obtain Theorem \ref{intropunct}.
\begin{theorem}
\label{posiffcoad}
The following are equivalent.
\begin{enumerate}[(i)]
\item $M_{\lambda}$ is a coadmissible $\wideparen{\mathcal{D}}(X)$-module.
\item The map $\theta_{\lambda}$ is an isomorphism.
\item $\lambda$ is of positive type.
\end{enumerate}
\end{theorem}
\begin{proof}
As before, we can assume that $x^\lambda$ generates $N_{\lambda}$ both as a $D$-module and as an $A[x^{-1}]$-module. Then we have
\begin{equation*}
\partial x^{-s} \cdot x^{\lambda}=(\lambda-s) x^{-s-1} \cdot x^{\lambda},
\end{equation*}
so that the associated $b$-function has $\lambda$ as its unique root.\\
\\
Thus (iii) implies (ii) by Theorem \ref{generalextn}, (ii) implies (i) by Proposition \ref{surjiffcoad}, and (i) implies (iii) by Proposition \ref{coadimpliespos}.
\end{proof}

\appendix
\section{Completed tensor products}
Let $\mathcal{U}$ be a locally convex $K$-algebra, and $V_1$ (resp. $V_2$) a locally convex right (resp. left) $\mathcal{U}$-module. We denote by $\rho: V_1\otimes_K V_2\to V_1\otimes_{\mathcal{U}} V_2$ the natural surjection. If both sides are equipped with their respective projective tensor product topologies, this is a strict surjection by definition. \\
We verify that the completed tensor product satisfies the following universal property. 
\begin{lemma}
\label{ctpuniv}
Let $\theta: V_1\times V_2\to W$ be a continuous $K$-bilinear $\mathcal{U}$-balanced map into a complete locally convex $K$-vector space $W$. Then there exists a unique continuous $K$-linear map $\alpha: V_1\widehat{\otimes}_{\mathcal{U}} V_2\to W$ such that $\theta$ factors as the composition of $\alpha$ with the canonical map $V_1\times V_2\to V_1\widehat{\otimes}_{\mathcal{U}} V_2$.
\end{lemma}
\begin{proof}
By \cite[Lemma 17.1]{SchneiderNFA}, there exists a unique continuous $K$-linear map $\alpha': V_1\otimes_K V_2\to W$ such that $\theta$ factors through $\alpha'$, and as $\theta$ is $\mathcal{U}$-balanced, this descends to the quotient $V_1\otimes_{\mathcal{U}}V_2$. As the surjection $\rho: V_1\otimes_K V_2\to V_1\otimes_{\mathcal{U}} V_2$ is strict by definition, it follows that the induced $K$-linear map $V_1\otimes_{\mathcal{U}} V_2\to W$ is also continuous when the tensor products are equipped with the projective tensor product topology. Since $W$ is assumed to be complete, this gives the desired continuous map from the completed tensor product to $W$.
\end{proof}
\begin{lemma}
\label{ctpprops}
\begin{enumerate}[(i)]
\item The natural morphism $\mathcal{U}\widehat{\otimes}_{\mathcal{U}} V_2\to \widehat{V_2}$ is an isomorphism.
\item The natural morphism 
\begin{equation*}
V_1\widehat{\otimes}_{\mathcal{U}} V_2\to \widehat{V_1}\widehat{\otimes}_{\widehat{\mathcal{U}}} \widehat{V_2}
\end{equation*}
is an isomorphism.
\item Suppose that $V_1$ is Fr\'echet, and write $q_1\leq q_2 \leq \dots$ for a family of defining semi-norms on $V_1$. Denote the Banach completion of $V_1$ with respect to $q_n$ by $V_{1,n}$. Then the natural morphism
\begin{equation*}
V_1\widehat{\otimes}_{\mathcal{U}} V_2\to \varprojlim_n \left(V_{1,n} \widehat{\otimes}_{\mathcal{U}} V_2\right)
\end{equation*}
is an isomorphism.
\item If $V_1\cong \varprojlim V_{1, n}$ and $V_2\cong \varprojlim V_{2, n}$ are both Fr\'echet as above, the natural morphism
\begin{equation*}
V_1\widehat{\otimes}_{\mathcal{U}} V_2\to \varprojlim \left(V_{1, n}\widehat{\otimes}_{\mathcal{U}} V_{2, n}\right)
\end{equation*}
is an isomorphism.
\end{enumerate}
\end{lemma}
\begin{proof}
The first two claims follow as usual from the universal property in Lemma \ref{ctpuniv}. For the last two claims, recall from the proof of \cite[Proposition 7.5]{SchneiderNFA} that the Hausdorff completion of any locally convex $K$-vector space $V$ can be constructed as 
\begin{equation*}
\widehat{V}\cong \varprojlim V/L,
\end{equation*}
where the limit ranges over all open lattices $L$ in $V$. Denote the unit ball of $V_1$ with respect to the semi-norm $q_n$ by $\mathcal{L}_n$. As open lattices in $V_1\otimes_{\mathcal{U}} V_2$ are given as images of $L_1\otimes_R L_2$ for $L_i$ open in $V_i$, we can take the limit over pairs $(\pi^i \mathcal{L}_n, L_2)$ where $i\in \mathbb{Z}$, $n\in \mathbb{N}$, and $L_2$ is open in $V_2$. We thus obtain
\begin{align*}
V_1\widehat{\otimes}_{\mathcal{U}} V_2&\cong \varprojlim_{L_1, L_2} \left(V_1\otimes_{\mathcal{U}} V_2/\rho(L_1\otimes_R L_2) \right)\\
& \cong \varprojlim_n \left(\varprojlim_{i\in \mathbb{Z}, L_2} V_1\otimes_{\mathcal{U}} V_2/\rho(\pi^i\mathcal{L}_n\otimes_R L_2) \right)\\
& \cong \varprojlim_n \left( V_{1, n}\widehat{\otimes}_{\mathcal{U}} V_2\right),
\end{align*} 
where the last isomorphism follows from the isomorphism (ii) above applied to the semi-normed space $(V_1, q_n)$ and the locally convex space $V_2$.\\
This establishes (iii), and the proof of (iv) is entirely analogous, noting that pairs of lattices $(\pi^i\mathcal{L}_n, \pi^j\mathcal{L}'_{n})$ form a cofinal system within the system of all pairs of open lattices, where $\mathcal{L}_n$ (resp. $\mathcal{L}'_n$) is the unit ball of $V_1$ (resp. $V_2$) with respect to the $n$th defining semi-norm. 
\end{proof}
\begin{lemma}
\label{assctp}
Let $\mathcal{U}$, $\mathcal{V}$ be two locally convex $K$-algebras. Let $V_1$ be a locally convex right $\mathcal{U}$-module, $V_2$ a locally convex $(\mathcal{U}, \mathcal{V})$-bimodule, and $V_3$ a locally convex left $\mathcal{V}$-module. Then the natural morphism
\begin{equation*}
(V_1\otimes_{\mathcal{U}}V_2)\otimes_{\mathcal{V}} V_3\to V_1\otimes_{\mathcal{U}}(V_2\otimes_{\mathcal{V}} V_3)
\end{equation*}
is an isomorphism of locally convex $K$-spaces, inducing an isomorphism
\begin{equation*}
(V_1\widehat{\otimes}_{\mathcal{U}}V_2)\widehat{\otimes}_{\mathcal{V}} V_3\cong V_1\widehat{\otimes}_{\mathcal{U}}(V_2\widehat{\otimes}_{\mathcal{V}}V_3).
\end{equation*}
\end{lemma}
\begin{proof}
First note that the associativity of $\otimes_K$ with the projective tensor product topology follows directly from the definition in terms of open lattices. Moreover, if $M_1\to M_2$ is a strict surjection, then $M_1\otimes_K V\to M_2\otimes_K V$ is a strict surjection for any locally convex $K$-spaces $M_1$, $M_2$, $V$.\\
\\ 
By definition, the surjection $V_1\otimes_K V_2\to V_1\otimes_{\mathcal{U}} V_2$ is strict, so
\begin{equation*}
(V_1\otimes_K V_2)\otimes_K V_3 \to (V_1\otimes_{\mathcal{U}} V_2)\otimes_K V_3
\end{equation*}
is a strict surjection by the above. Thus the projective tensor product topology on $(V_1\otimes_{\mathcal{U}}V_2)\otimes_{\mathcal{V}}V_3$ is equivalent to the quotient topology induced by the surjection
\begin{equation*}
(V_1\otimes_K V_2)\otimes_K V_3\to (V_1\otimes_{\mathcal{U}}V_2)\otimes_{\mathcal{V}} V_3.
\end{equation*}
The analogous statement holds for $V_1\otimes_{\mathcal{U}}(V_2\otimes_{\mathcal{V}} V_3)$.\\
Using associativity of the projective tensor product over $K$, we thus obtain that the natural bijection
\begin{equation*}
(V_1\otimes_{\mathcal{U}}V_2)\otimes_{\mathcal{V}} V_3\to V_1\otimes_{\mathcal{U}}(V_2\otimes_{\mathcal{V}} V_3)
\end{equation*}
is an isomorphism of locally convex $K$-spaces, and applying Lemma \ref{ctpprops}.(ii) yields the result for the completions.
\end{proof}

We now verify that the coadmissible tensor product $\wideparen{\otimes}$ defined in \cite[Lemma 7.3]{Ardakov1} is just $\widehat{\otimes}$. 
\begin{lemma}
Let $\mathcal{U}$ be a left Noetherian Banach algebra. Let $M$ be a finitely generated left $\mathcal{U}$-module, equipped with its canonical Banach topology, and let $W$ be any locally convex left $\mathcal{U}$-module. Then any $\mathcal{U}$-linear map $\phi: M\to W$ is continuous.
\end{lemma}
\begin{proof}
Let $\mathcal{U}^\circ$ denote the unit ball of $\mathcal{U}$. Let $L$ be an open lattice in $W$. Since the action map $\mathcal{U}\otimes_K W\to W$ is continuous, there exists an open lattice $L'$ of $W$ such that $\mathcal{U}^\circ\otimes_R L'$ maps into $L$, and $L$ contains the lattice $\mathcal{U}^\circ L'$ that is $\mathcal{U}^\circ$-stable. \\
Let $m_1, \dots, m_r$ be a finite generating set of $M$. There exists some integer $n$ such that $\pi^n\phi(m_i)\in \mathcal{U}^\circ L'$ for each $i$, and thus 
\begin{equation*}
\sum_i \pi^n\mathcal{U}^\circ \phi(m_i)\subseteq \mathcal{U}^\circ L'\subseteq L
\end{equation*}
by $\mathcal{U}$-linearity of $\phi$. Thus $\phi$ is continuous.
\end{proof}
\begin{lemma}
\label{fgctp}
Let $\mathcal{U}=\varprojlim \mathcal{U}_n$ be a left Fr\'echet--Stein algebra and let $M$ be a left coadmissible $\mathcal{U}$-module. Then the canonical Banach topology on the finitely generated $\mathcal{U}_n$-module $M_n:=\mathcal{U}_n\otimes_{\mathcal{U}} M$ is equivalent to the projective tensor product topology.\\
In particular, $\mathcal{U}_n\otimes_{\mathcal{U}}M\cong \mathcal{U}_n\widehat{\otimes}_{\mathcal{U}}M$.
\end{lemma}
\begin{proof}
Let $T_1$ denote the projective tensor product topology on $M_n$, and $T_2$ the canonical $\mathcal{U}_n$-topology.\\
The natural bijection $(M_n, T_1)\to (M_n, T_2)$ is continuous by the universal property of projective tensor products, and its inverse is continuous by the previous lemma.\\
In particular, $\mathcal{U}_n\otimes_{\mathcal{U}} M$ with the projective tensor product topology is already complete.
\end{proof}
\begin{corollary}
Let $\mathcal{U}$ and $\mathcal{V}$ be left Fr\'echet--Stein algebras. Let $M$ be a left coadmissible $\mathcal{V}$-module and let $P$ be a $\mathcal{U}$-coadmissible $(\mathcal{U}, \mathcal{V})$-bimodule as defined in \cite[Definition 7.3]{Ardakov1}. Then the coadmissible tensor product
\begin{equation*}
P\wideparen{\otimes}_{\mathcal{V}}M:=\varprojlim \left((\mathcal{U}_n\otimes_{\mathcal{U}} P)\otimes_{\mathcal{V}} M\right)
\end{equation*}
is isomorphic to the completed tensor product $P\widehat{\otimes}_{\mathcal{V}} M$.
\end{corollary}
\begin{proof}
The coadmissible module $P$ is Fr\'echet, where a family of defining semi-norms is obtained from the isomorphism $P\cong \varprojlim (\mathcal{U}_n\otimes_{\mathcal{U}} P)$. Now apply Lemma \ref{ctpprops}.(iii) to get
\begin{align*}
P\widehat{\otimes}_{\mathcal{V}} M&\cong \varprojlim \left((\mathcal{U}_n\otimes_{\mathcal{U}} P) \widehat{\otimes}_{\mathcal{V}} M \right) \\
& \cong \varprojlim \left((\mathcal{U}_n\otimes_{\mathcal{U}} P)\widehat{\otimes}_{\mathcal{V}_n}(\mathcal{V}_n\widehat{\otimes}_{\mathcal{V}} M) \right).
\end{align*}
By Lemma \ref{fgctp}, $\mathcal{V}_n\widehat{\otimes}_{\mathcal{V}} M\cong \mathcal{V}_n\otimes_{\mathcal{V}} M$, so that
\begin{equation*}
P\widehat{\otimes}_{\mathcal{V}} M\cong \varprojlim \left((\mathcal{U}_n\otimes_{\mathcal{U}} P)\widehat{\otimes}_{\mathcal{V}_n} (\mathcal{V}_n\otimes_{\mathcal{V}} M) \right).
\end{equation*} 
Now both $\mathcal{U}_n\otimes_{\mathcal{U}}P$ and $\mathcal{V}_n\otimes_{\mathcal{V}} M$ are equipped with their canonical Banach topologies as finitely generated modules, so by the same argument as in Lemma \ref{fgctp}, their tensor product topology is equivalent to the canonical Banach topology of the finitely generated $\mathcal{U}_n$-module
\begin{equation*}
(\mathcal{U}_n\otimes_{\mathcal{U}}P)\otimes_{\mathcal{V}_n}(\mathcal{V}_n\otimes_{\mathcal{V}} M).
\end{equation*}
In particular, this is already complete, so that 
\begin{equation*}
P\widehat{\otimes}_V M\cong \varprojlim \left((\mathcal{U}_n\otimes_{\mathcal{U}}P)\otimes_{\mathcal{V}} M \right)
\end{equation*}
as required.
\end{proof}

\section{Positive type \`a la Kedlaya}
\begin{definition}[{see \cite[Definition 13.1.1]{Kedlaya}}]
\label{kedlayadef}
The type of $\lambda \in K$, denoted $\mathrm{type}(\lambda)$, is the radius of convergence of the formal power series
\begin{equation*}
\sum_{i\geq 0, i\neq \lambda} \frac{x^i}{\lambda-i}.
\end{equation*}
\end{definition}
In particular, $\mathrm{type}(\lambda)>0$ if and only if there exists some integer $r$ such that 
\begin{equation*}
\frac{\pi^{ir}}{\lambda -i}\to 0 \ \text{as} \ i\to \infty.
\end{equation*}
We now verify that $\lambda$ is of positive type as defined in Definition \ref{defpostype} if and only if $\mathrm{type}(\lambda)>0$.
\begin{lemma}
\label{kedlayaseries}
Let $\lambda \in K$ and $\lambda \notin \mathbb{Z}_{\geq 0}$. Then $\mathrm{type}(\lambda)>0$ if and only if there exists some integer $r$ such that
\begin{equation*}
\frac{\pi^{ir}}{\prod_{j=0}^{i-1}(\lambda-j)}\to 0 \ \mathrm{as} \ i\to \infty.\ \ \ \ \ \ (*)
\end{equation*}
\end{lemma}
\begin{proof}
By \cite[Lemma 13.1.6]{Kedlaya}, we have the following equality of formal power series:
\begin{equation*}
\sum_{i\geq 0} \frac{x^i}{\lambda(1-\lambda)\dots (i-\lambda)}=e^x\cdot \sum_{i \geq 0} \frac{(-x)^i}{i!}\frac{1}{\lambda-i}.\ \ \ \  (**)
\end{equation*} 
Suppose $\mathrm{type}(\lambda)>0$, so that we have some $r\geq 0$ such that
\begin{equation*}
\frac{\pi^{ir}}{\lambda-i}\to 0.
\end{equation*}
Then for $|x|\leq |\pi|^r\cdot |p|^{1/(p-1)}$, we have 
\begin{equation*}
\left| \frac{x^i}{i!(\lambda-i)}\right|\leq \left|\frac{\pi^{ir}}{\lambda-i}\right|\cdot \left|\frac{p^{i/(p-1)}}{i!}\right|\to 0.
\end{equation*}
So the right hand side of the equation converges for $|x|\leq |\pi|^m$, where $|\pi|^m=|\pi|^r\cdot |p|^{1/(p-1)}$. Thus the same is true for the left hand side, and hence
\begin{equation*}
\left| \frac{\pi^{im}}{\lambda \cdot (\lambda-1) \dots (\lambda-i+1)}\right|\leq \left| \frac{\pi^{(i-1)m}}{\lambda \cdot (\lambda-1)\dots (\lambda-i+1)}\right|\to 0.
\end{equation*}
Thus $\mathrm{type}(\lambda)>0$ implies that $(*)$ is satisfied.\\
\\
For the converse, first note that if $\lambda=-1$, we are done: $\mathrm{type}(-1)>0$, and $-1$ satisfies $(*)$ by the above. So from now on assume $\lambda\neq -1$, and suppose that for some integer $m$,
\begin{equation*}
\frac{\pi^{im}}{\prod_{j=0}^{i-1} (\lambda-j)}\to 0.
\end{equation*}
Write $\mu=\lambda+1$. Mutliplying the above by $1/\mu$, we know that 
\begin{equation*}
\frac{\pi^{im}}{\mu \cdot (\mu-1)\dots (\mu-i)}\to 0.
\end{equation*}
Thus the formal power series on the left hand side of $(**)$ has positive radius of convergence when $\lambda$ is replaced by $\mu$, and (as $e^{-x}$ also has positive radius of convergence) it follows that the right hand side has positive radius of convergence, giving
\begin{equation*}
\left|\frac{\pi^{ir}}{\mu-i}\right|\leq \left|\frac{\pi^{ir}}{i!(\mu-i)}\right|\to 0
\end{equation*}
for some integer $r$.\\
Thus
\begin{equation*}
\left| \frac{\pi^{ir}}{\lambda-i}\right|=\left| \frac{\pi^{-r} \cdot \pi^{(i+1)r}}{\lambda+1-(i+1)}\right| =|\pi|^{-r}\cdot \left|\frac{\pi^{(i+1)r}}{\mu-(i+1)}\right|\to 0,
\end{equation*}
and $\mathrm{type}(\lambda)>0$. 
\end{proof}

\end{document}